\documentclass[11pt]{amsart}
\usepackage{epsfig,color}
\usepackage{blindtext}
\usepackage{graphicx}
\usepackage{enumitem}
\usepackage{url}
\usepackage{amssymb}
\usepackage{graphicx,import}
\usepackage{comment}
\usepackage{esint}
\usepackage{xypic}

\usepackage[margin = 1.4in] {geometry}

\usepackage{hyperref}

\newtheorem{thm}{Theorem}[section]
\newtheorem{cor}[thm]{Corollary}
\newtheorem{lem}[thm]{Lemma}

\numberwithin{equation}{section}

\theoremstyle{theorem}

\theoremstyle{definition}

\newtheorem*{remark*}{Remark}

\numberwithin{equation}{section}

\author{Elham Matinpour}
\address{Department of Mathematics, Johns Hopkins University, 3400 N. Charles Street, Baltimore, MD 21218}
\email{ematinp1@jhu.edu}

\title{First eigenvalue and nodal domains of the drift Laplacian on symmetric self-shrinkers in $\mathbb{R}^3$}

\begin{document}

\maketitle

\begin{abstract}
Consider $\mathbb{R}^3$ equipped with the Euclidean metric and the Gaussian measure $d\mu =e^{\frac{-\vert x\vert^2}{4}}dx$. Let $\Sigma \subset \mathbb{R}^3$ be a complete embedded self-shrinker with the induced metric and weighted measure, and let $\lambda_1$ denote the first eigenvalue of the drift Laplacian $\Delta_{\mu}=\Delta -\frac{1}{2}\langle x,\nabla .\rangle$ in the weighted space $L^2(\Sigma, e^{\frac{-\vert x\vert^2}{4}}dx)$. Inspired by Choe and Soret’s estimate, \cite{choe2009first}, of the first eigenvalue of the Laplacian on symmetric minimal surfaces in $\mathbb{S}^3$, we prove that $\lambda_1=\frac{1}{2}$ for self-shrinkers invariant under the dihedral group $\mathbb{D}_{g+1}$ or the prismatic group $\mathbb{D}_{g+1}\times \mathbb{Z}_2$. In particular, this holds for self-shrinkers constructed by Kapouleas, Kleene, and Moller \cite{kapouleas2018mean}, Buzano, Nguyen, and Schulz \cite{buzano2025noncompact}, Ketover \cite{ketover2024self}, and Ilmanen and White \cite{ilmanen2024fattening} confirming a universal spectral property tied to their symmetry.
\end{abstract}

\section{Introduction} A surface $\Sigma \subset \mathbb{R}^3$ is a self-shrinker if its mean curvature $H$ satisfies $H=-\frac{1}{2}x^{\perp}$, where $x$ is the position vector and $x^{\perp}$ is its normal component. Self-shrinkers are fundamental in mean curvature flow, as they model singularities and evolve homothetically under the flow $\Sigma_t=\sqrt{-t}\Sigma ,\ t\in (-\infty ,0)$, satisfying $\frac{dx}{dt}=H$. Equivalently, a self-shrinker is:\\
1. A minimal surface in $\mathbb{R}^3$ with the conformally flat metric $g=e^{-\frac{\vert x\vert^2}{4}}\langle .,.\rangle$.\\
2. A critical point of the Gaussian area functional:
$$
F(\Sigma)=(4\pi)^{-\frac{3}{2}}\int_{\Sigma} e^{-\frac{\vert x\vert^2}{4}}dx
$$
where $dx$ is the area element on $\Sigma$ (see \cite{huisken1990asymptotic}, \cite{colding2012generic}, \cite{colding2012smooth}). These properties link self-shrinkers to minimal surfaces, suggesting that techniques from minimal surface theory can be adapted to study their geometry and spectral properties.
\medskip

In minimal surface theory, the Laplacian’s eigenvalues reveal geometric and topological features. For a minimal $k$-submanifold $M\subset \mathbb{R}^{n+1}$, the coordinate functions $x_1,\cdots ,x_{n+1}$ are harmonic ($\Delta_M x_i=0$). If $M\subset \mathbb{S}^n \subset \mathbb{R}^{n+1}$ is a $k$-dimensional minimal submanifold of the unit $n$-sphere, these functions are eigenfunctions of $\Delta_M$ with eigenvalue $k$, implying the first non-trivial eigenvalue satisfies $\lambda_1\leq k$. Yau \cite{yau1982problem} conjectured that for compact, embedded minimal hypersurfaces in $\mathbb{S}^{n+1}$, $\lambda_1=n$, a conjecture still open. Choe and Soret \cite{choe2009first} proved $\lambda_1=2$ for compact embedded minimal surfaces in $\mathbb{S}^3$ invariant under finite reflection groups, confirming Yau’s conjecture for known examples.
\medskip

By analogy, we study the drift Laplacian $\Delta_{e^{-\frac{\vert x\vert^2}{4}}}=\Delta -\frac{1}{2}\langle x,\nabla .\rangle$, defined on a self-shrinker $\Sigma \subset \mathbb{R}^3$ with the Gaussian weight $e^{-\frac{\vert x\vert^2}{4}}$, in the weighted space $L^2(\Sigma , e^{-\frac{\vert x\vert^2}{4}}dx)$. This operator, associated with the Gaussian measure, governs the stability and spectral properties of self-shrinkers. Motivated by recent constructions of embedded self-shrinkers with arbitrary genus and symmetries, such as dihedral $\mathbb{D}_{g+1}$ or prismatic $\mathbb{D}_{g+1} \times \mathbb{Z}_2$ symmetries \cite{kapouleas2018mean}, \cite{buzano2025noncompact}, \cite{ketover2024self}, \cite{ilmanen2024fattening}, we investigate the first eigenvalue $\lambda_1$ and its eigenfunctions. We prove that $\lambda_1=\frac{1}{2}$ for such symmetric self-shrinkers, a result that reflects their geometric and stability characteristics, paralleling the spectral insights from minimal surfaces.

\section{preliminaries}
The drift Laplacian $\Delta_{-\frac{\vert x\vert^2}{4}}$ is a linear operator denoted by $\mathcal{L}$ on $\mathbb{R}^n$ 
$$
\mathcal{L}=\Delta_{-\frac{\vert x\vert^2}{4}}(\cdot) =e^{\frac{\vert x\vert^2}{4}} \text{div}(e^{-\frac{\vert x\vert^2}{4}}\nabla \cdot)
$$
This is so-called Ornstein-Uhlenbeck operator, associated with the Gaussian measure $d\mu =e^{-\frac{\vert x\vert^2}{4}} dx$. The $\mathcal{L}$ operator was introduced on self-shrinkers by Colding-Minicozzi \cite{colding2012generic}. They showed that it is self-adjoint in $L^2(\mathbb{R}^n,d\mu)$. We compute
$$
\mathcal{L}u=e^{\frac{\vert x\vert^2}{4}} \nabla \cdot (e^{-\frac{\vert x\vert^2}{4}}\nabla u) = e^{\frac{\vert x\vert^2}{4}} (e^{-\frac{\vert x\vert^2}{4}} \Delta u + \nabla(e^{-\frac{\vert x\vert^2}{4}})\cdot \nabla u)
$$
and 
$$
\nabla(e^{-\frac{\vert x\vert^2}{4}}) = -\frac{1}{2} x e^{-\frac{\vert x\vert^2}{4}}
$$
hence,
$$
\mathcal{L}u= \Delta u-\frac{1}{2}x\cdot \nabla u
$$
So, $\mathcal{L}= \Delta -\frac{1}{2} \langle x, \nabla (\cdot)\rangle$ is a drifted Laplacian tied to the Gaussian density. The eigenvalues of this operator are the solutions to the problem $\mathcal{L}u+\lambda u=0$, where $\lambda \geq 0$ and eigenvalues are discrete due to the weight’s confinement effect:
$$
0=\lambda_0<\lambda_1 \leq \lambda_2 \leq \cdots
$$
Let $M$ be a $k$-dimensional complete self-shrinker in $\mathbb{R}^n$. Then by $H_M=\frac{-1}{2}x^{\perp}$, we have $\Delta x=H_M=-\frac{1}{2}x^N$ for any $x=(x_1,\cdots ,x_n)\in \mathbb{R}^n$, then
$$
\mathcal{L}x_i=\Delta x_i -\frac{1}{2} \langle x,\nabla x_i\rangle =-\frac{1}{2}\langle x^N ,E_i\rangle -\frac{1}{2}\langle x, (E_i)^T\rangle =-\frac{1}{2}x_i,
$$
where $\{ E_i\}_{i=1}^{n}$ is a standard basis of $\mathbb{R}^n$, and $(\cdots )^N$ and $(\cdots)^T$ denote the orthogonal projection into the normal and tangent bundles, $NM$ and $TM$. It follows that the coordinate functions are eigenfunctions associated with eigenvalue $\frac{1}{2}$. The analogous conjecture is that the first eigenvalue of the drift Laplacian $\mathcal{L}$ on self-shrinkers is equal to $\frac{1}{2}$. 
\medskip

In \cite{ding2013volume}, Ding and Xin in a manner analogous to the arguments in \cite{choi1983first} used a Reilly type formula for $\mathcal{L}$ to estimate its first eigenvalue. They proved that the first eigenvalue is bounded below by $\frac{1}{4}$ for compact embedded self-shrinkers in $\mathbb{R}^n$. Define the first (Neumann) eigenvalue $\lambda_1$ of the operator $\mathcal{L}$ on complete compact self-shrinkers $M^n$ in $\mathbb{R}^{n+1}$ by
$$
\lambda_1=\underset{f\in C^{\infty}(M)}{\inf} \Big \{ \int_M \vert \nabla f\vert^2 d\mu ;\ \ \int_Mf^2d\mu =1,\  \int_Mfd\mu =0  \Big \}
$$
As we shown the coordinate functions are eigenfunctions associated with the eigenvalue $\frac{1}{2}$, and so $0<\lambda_1 \leq \frac{1}{2}$, for any complete properly immersed self-shrinker $M^n$. And, by [Theorem 1.3. \cite{ding2013volume}],  $\frac{1}{4}\leq \lambda_1 \leq \frac{1}{2}$, for compact and embedded $M^n$.


\section{A Courant-type theorem and Nodal sets of the drift Laplacian}
To estimate the first eigenvalue of the drift Laplacian we rely on the nodal properties of associated eigenfunctions. For the standard Laplacian, Courant’s theorem asserts that the first non-trivial eigenfunction (associated with the first neumann eigenvalue) divides the complete manifold into two nodal domains. As shown by R. Chen, J. Mao and C. Wu, \cite{chen2025eigenfunctions}, the Courant nodal domain still holds for a well-defined weighted Laplacian operator on an $n$-dimensional Riemannian manifold $M^n$. Here, we use their result to extend this property to the drift operator on self-shrinkers in Euclidean spaces. 
\medskip

Let $(M^n, \langle.,.\rangle)$, $n\geq 2$, be an $n$-dimensional complete Riemannian manifold with the metric $\langle .,.\rangle$. Consider $\phi \in C^{\infty}(M)$, a smooth real-valued function defined on $M$, and define the elliptic operator,
$$
\Delta_{\phi}:=\Delta -\langle \nabla \phi, \nabla .\rangle
$$
that is a weighted Laplacian with the weight $\phi$ defined on domains $\Omega$ in $M$, and $\nabla ,\ \Delta$ are the gradient and Laplacian operators on $M$, respectively. We may call $\Delta_{\phi}$ the $\phi$-weighted Laplacian on $M$, and we assume that $\phi$ is chosen so that the associated weighted Laplacian operator is well-defined on $M$. This is a self-adjoint operator with respect to the $d\mu (=\phi dx)$-measure. 
Now suppose that $\Omega$ is a bounded domain and consider the Dirichlet eigenvalue problem of the $\phi$-weighted Laplacian as follows
\begin{equation} \label{Dirichlet problem for wei.Laplacian}
\begin{cases}
    \Delta_{\phi} u+\lambda u=0\ \ \ \text{in}\ \Omega \subset M\\
    u=0 \hspace{1.8cm} \text{on}\ \partial \Omega
\end{cases}
\end{equation}
 The self-adjoint elliptic operator $\Delta_{\phi}$ in ($\ref{Dirichlet problem for wei.Laplacian}$) only has discrete spectrum, and all the eigenvalues in this discrete spectrum can be listed non-decreasingly as
$$
0<\lambda_1(\Omega)<\lambda_2(\Omega)\leq \lambda_3(\Omega) \leq \cdots \rightarrow \infty
$$
where the eigenvalues are repeated with respect to their (finite) multiplicities. By the standard variational principles we obtain
$$
\lambda_k=\underset{f\in W_0^{1,2}(\Omega)}{\inf} \Big \{ \int_{\Omega} \vert \nabla f\vert^2 d\mu ;\ \ \int_{\Omega}f^2d\mu =1,\  \int_{\Omega}ff_id\mu =0   \Big \}
$$
where $f_i,\ i=1,\cdots ,k-1$ denotes an eigenfunction of $\lambda_i(\Omega)$, and $W_0^{1,2}(\Omega)$ is a Sobolev space. And,
$$
\lambda_1=\underset{f\in W_0^{1,2}(\Omega)}{\inf} \Big \{ \int_{\Omega} \vert \nabla f\vert^2 d\mu ;\ \ \int_{\Omega}f^2d\mu =1,\  \int_{\Omega}fd\mu =0   \Big \}
$$
If the bounded domain $\Omega\subset M$ has no boundary, then the closed eigenvalue problem of $\Delta_{\phi}$ is
\begin{equation}\label{closed eignvs for wei.Laplacian}
\Delta_{\phi} u+\lambda u=0\ \ \ \text{in}\ \Omega
\end{equation}
Then $\Delta_{\phi}$ still has discrete spectrum listed as
$$
0=\lambda_0 (\Omega) < \lambda_1(\Omega)\leq \lambda_2(\Omega)\cdots \rightarrow \infty
$$
with multiplicity counted. We also have
$$
\lambda_k=\underset{f\in W^{1,2}(\Omega)}{\inf} \Big \{ \int_{\Omega} \vert \nabla f\vert^2 d\mu ;\ \ \int_{\Omega}f^2d\mu =1,\  \int_{\Omega}ff_id\mu =0   \Big \}
$$
where $f_i,\ i=1,\cdots ,k-1$ denotes an eigenfunction of $\lambda_i(\Omega)$, and define $\lambda_1$ accordingly. 
\begin{thm} \label{thm: courant}[Chen-Mao-Wu, \cite{chen2025eigenfunctions}] 
    The Dirichlet eigenvalue problem (\ref{Dirichlet problem for wei.Laplacian}) on a regular domain $\Omega$ in $M$ has a non-decreasing sequence  of solutions 
   $0<\lambda_1(\Omega)<\lambda_2(\Omega)\leq \lambda_3(\Omega) \leq \cdots$ with associated eigenfunctions $f_1,f_2,f_3,\cdots,$ which form a complete orthogonal basis of $L_{\phi}^2(\Omega)$. Moreover, the number of nodal domains of $f_k$ is less than or equal to $k$. A very similar result holds for the closed eigenvalue problem (\ref{closed eignvs for wei.Laplacian}), where the sequence of solutions is $0=\lambda_0 (\Omega) < \lambda_1(\Omega)\leq \lambda_2(\Omega)\cdots$ with the number of nodal domains of an eigenfunction of the $k$-th eigenvalue is less than or equal  to $k+1$.
\end{thm}
In particular when $n=2$, they show that the nodal sets of the weighted Laplacian share the same properties as those of the standard Laplacian, see [Theorem 2.5, \cite{cheng1976eigenfunctions}] for the standard Laplacian.
\begin{thm} [Theorem 1.4, \cite{chen2025eigenfunctions}]\label{thm: Props. of nodal lines}
    Assume that $M^2$ is a smooth Riemannian 2-manifold without boundary (not necessarily compact). If a smooth function $f\in C^{\infty} (M)$ is an eigenfunction of the weighted Laplacian $\Delta_{\phi}$, then except for a few isolated points, the set of nodes of $f$ forms a smooth curve. Moreover, one has:\\
    $\bullet$ When the nodal lines meet, they form an equiangular system.\\
    $\bullet$ The nodal lines consist of a number of $C^2$-immersed closed smooth curves.
\end{thm}

Now suppose that $\Sigma^n$ is a self-shrinker hypersurface in $\mathbb{R}^{n+1}$ and let $\phi =-\frac{\vert x\vert^2}{4}$ the Gaussian density, and in particular, $\Delta_{\phi}=\mathcal{L}$ the drift Laplacian. If $\Sigma$ is non-compact then restrict the domain of the drift Laplacian $\mathcal{L}$ to the $L_{\phi}^2(\Sigma)$-space with respect to the $\phi$-weighted area element $d\mu$, as this ensures a well-defined spectrum.
 Although $\Sigma$ is non-compact, we assume that the Gaussian area of $\Sigma$ is finite, and in particular, $\Sigma$ has a polynomial volume growth which ensures the discreteness of the spectrum. Then we consider the closed eigenvalue problem of $\mathcal{L}$ as
 $$
 \mathcal{L}u+\lambda u=0
$$
where $u\in L_{\phi}^2(\Sigma)$. Assuming then $\Sigma$ is connected, the nodal set of $f_1$ is a co-dimension one submanifold, dividing $\Sigma$ into two regions, which we use to bound $\lambda_1$ via domain decomposition.
\section{Two-piece property of self-shrinkers in $\mathbb{R}^n$} 
Note that for $\phi \in L_{\mu}^2(\Sigma)$, if $\mathcal{L} \phi +\lambda_1 \phi =0$ then $\phi$ has exactly two nodal domains, but an eigenfunction with exactly two nodal domains is not necessarily the first eigenfunction. If $\frac{1}{2}$ is the first eigenvalue of $\mathcal{L}$ on $\Sigma$ then we can apply the Courant's nodal theorem to a linear function $a_1x_1+a_2x_2+a_3x_3$ on $\Sigma$, that is, any plane passing through the origin will cut $\Sigma$ into two connected pieces. This is an analogy to two-piece property proved by Ros, \cite{ros1995two}, for compact embedded minimal surfaces in $\mathbb{S}^3$, and later, generalized to compact embedded minimal hypersurfaces in $\mathbb{S}^n$ by Choe and Soret, \cite{choe2009first}. In \cite{brendle2016embedded}, Brendle gives the parallel property for embedded self-shrinkers in $\mathbb{R}^3$. In this section, we state the work of Brendle's with a sketch of the proof for the convenience of the reader.
\medskip

Given a unit vector $v\in \mathbb{S}^2$ , define the hyperplane $\Pi_v=\{ p\in \mathbb{R}^{3};\ \langle v,p\rangle =0\}$, and half-spaces $H_v^+=\{ p\in \mathbb{R}^3;\ \langle v,p\rangle >0\}$ and $H_v^-=\{ p\in \mathbb{R}^3;\ \langle v,p\rangle <0\}$.
\begin{thm} \label{thm: 2-piece prop}
    Any hyperplane $\Pi_v$ passing through the origin in $\mathbb{R}^{3}$ divides a complete, embedded, self-shrinker surface $\Sigma$ of $\mathbb{R}^{3}$, except $\Pi_v$ itself, into two connected pieces.
\end{thm}
\begin{proof}
The proof is trivial for self-shrinkers $\Pi_u$, where $u (\neq v) \in \mathbb{S}^2$, hence we assume that $\Sigma$ is not a hyperplane passing through the origin. Suppose that $\Pi_v$ divides $\Sigma$ into more than two connected components, and assume that $\Sigma \cap H_v^+$ is not connected. Let $\Sigma_1$ be a connected components of $\Sigma \cap H_v^+$ and let $\Sigma_2=(\Sigma \cap H_v^+)\setminus \Sigma_1$. Obviously $\Sigma_2 \neq \emptyset$. Denote the connected components of $\mathbb{R}^{3}\setminus \Sigma$ by $U_1$ and $U_2$, and assume that $U_1$ is unbounded. For $k$ sufficiently large, denote by $\mathcal{E}_k$ the set of all embedded disks $S\subset \bar{U}_1\cap \{ \vert x\vert \leq 2k\}$ with the property that $\partial S=\partial \Sigma_1$. Choose a cut-off function $\psi_k:[0,\infty)\to [0,1]$ satisfying $\psi_k=0$ on $[0,k]$ and $\psi_k'(2k)>k$. Consider the functional
$$
\mathcal{F}_k(S)=\int_S e^{-\frac{\vert x\vert^2}{4}+\psi_k(\vert x\vert)} dx
$$
for $S\in \mathcal{E}_k$. $\mathcal{F}_k$ is the area functional for the conformal metric $e^{-\frac{\vert x\vert^2}{4}+\psi_k(\vert x\vert)}\delta_{ij}$. For $k$ sufficiently large, the region $\bar{U}_1\cap \{ \vert x\vert \leq 2k\}$ is a mean convex domain with respect to this conformal metric. Then by \cite{meeks1982existence}, there exists a smooth embedded surface $\Sigma_k\in \mathcal{E}_k$ which minimizes the functional $\mathcal{F}_k$. Then it follows that
$$
\sup_k \mathcal{F}_k(\Sigma_k)<\infty
$$
Then
$$
\sup_k \int_{\Sigma_k} e^{-\frac{\vert x\vert^2}{4}} dx<\infty
$$
Using the first variation formula, we have $H=\frac{1}{2}\langle x,\nu \rangle$ on $\Sigma_k\cap \{ \vert x\vert \leq k\}$. Finally, the stability inequality for the minimizer $\Sigma_k$ implies that
$$
0\leq -\int_{\Sigma_k} e^{-\frac{\vert x\vert^2}{4}}fLf dx
$$
for $f\in C_c^{\infty}(\bar{\Sigma}_k),\ f\vert_{\partial \Sigma_1}\equiv 0$. Employing an argument similar to  [Proposition 5, \cite{brendle2016embedded}], we get a uniform upper bound to $\vert A\vert^2$,
\begin{equation}\label{eqn: lim sup weigh. A 0}
\begin{split}
\lim_{k\to \infty}\sup \int_{\Sigma_k\cap \{ \vert x\vert \leq \sqrt{k}\}}\vert A\vert^2  e^{-\frac{\vert x\vert^2}{4}} dx =0
\end{split}
\end{equation}
By [Theorem 3, \cite{schoen1981regularity}] then
$$
\lim_{k\to \infty}\sup_{\Sigma_k \cap W}\vert A\vert^2 <\infty
$$
for every compact $W\subset \mathbb{R}^{3}\setminus \partial \Sigma_1$. The surfaces $\Sigma_k$ converge in $C_{loc}^{\infty}(\mathbb{R}^{3}\setminus \partial \Sigma_1)$ to a smooth surface $\tilde{\Sigma}_1 \subset \mathbb{R}^{3}\setminus \partial \Sigma_1$ which satisfies the shrinker equation $H=\frac{1}{2}\langle x,\nu\rangle$. Then by (\ref{eqn: lim sup weigh. A 0}), we conclude that $\tilde{\Sigma}_1$ is totally geodesic. \\
 If $U_2$ is not unbounded, then consider the set $\mathcal{C}$ of all embedded disks $S\subset \bar{U}_2$ with boundary $\partial S=\partial \Sigma_1$, and note that $\mathcal{C}\neq \emptyset$. Consider the area functional for the conformal metric $e^{-\frac{\vert x\vert^2}{4}}\delta_{ij}$,
$$
\mathcal{F}= \int_{\Sigma}e^{-\frac{\vert x\vert^2}{4}} dx
$$
Since $\bar{U}_2$ is a mean convex domain with respect to this metric, by \cite{meeks1982existence} there exists a smooth embedded surface $\tilde{\Sigma}_2\in \mathcal{C}$ which minimizes the functional $\mathcal{F}$. Then by the first variation formula, $\tilde{\Sigma}_2$ satisfies $H=\frac{1}{2}\langle x,\nu\rangle$. And, by the stability inequality,
$$
0\leq -\int_{\tilde{\Sigma}_2}e^{-\frac{\vert x\vert^2}{4}}fLf dx
$$
for $f\in C^{\infty}_0(\Sigma_1)$, then one can find that,
$$
\int_{\tilde{\Sigma}_2}\vert A\vert^2 e^{-\frac{\vert x\vert^2}{4}}\langle a,x\rangle^2 dx=0
$$
Then, $\tilde{\Sigma}_2$ is also totally geodesic, and thus $\langle x,\nu \rangle =0$ on $\tilde{\Sigma}_2$. Note that, since $\Sigma$ is not totally geodesic the strict maximum principle then implies that neither $\tilde{\Sigma}_1$ nor $\tilde{\Sigma}_2$ can touch $\Sigma$. Clearly, $\tilde{\Sigma}_1$ and $\tilde{\Sigma}_2$ are contained in the plane $\Pi_v$. By $\tilde{\Sigma}_1\subset U_1$ and $\tilde{\Sigma}_2\subset U_2$ then we conclude that $\tilde{\Sigma}_1$ and $\tilde{\Sigma}_2$ are disjoint. Note that they share the same boundary $\partial \Sigma_1$. Hence, $(\tilde{\Sigma}_1 \cup \tilde{\Sigma}_2 \cup \partial \Sigma_1) \subset \Pi_v$, is open and closed, and thus 
$$
(\tilde{\Sigma}_1 \cup \tilde{\Sigma}_2 \cup \partial \Sigma_1) = \Pi_v
$$
If $U_2$ is unbounded, the same argument as covered for $U_1$ applies to obtain another minimizing surface $\tilde{\Sigma}_2 \subset U_2$ which is totally geodesic, and it shares the same boundary $\partial \Sigma_1$. The same argument still holds, and $(\tilde{\Sigma}_1 \cup \tilde{\Sigma}_2 \cup \partial \Sigma_1) = \Pi_v$.
But this contradicts our original assumption because $\Sigma \cap \Pi_v =\partial \Sigma_1 \cup \partial \Sigma_2$ and $\partial \Sigma_2\neq \emptyset$. Therefore, $\Sigma \cap \Pi_v^+$ is connected. Similarly $\Sigma \cap \Pi_v^-$ is connected as well. 

\end{proof}
\section{self-shrinkers in $\mathbb{R}^3$}
In this section, we survey known complete self-shrinkers in $\mathbb{R}^3$, emphasizing their symmetries to provide context for our spectral analysis of the drift Laplacian in subsequent sections. The simplest complete embedded self-shrinkers are flat planes through the origin, invariant under $O(3)$, the sphere of radius 2 centered at the origin, also $O(3)$-invariant, cylinders of radius $\sqrt{2}$, invariant under $\mathbb{S}^1\times \mathbb{R}$ (rotations around the axis and translations). Brendle \cite{brendle2016embedded} proved that these are the only complete embedded self-shrinkers of genus zero. However, Drugan \cite{drugan2015immersed} constructed a rotationally symmetric, embedded self-shrinker of genus zero, topologically a sphere, which is rotationally symmetric and non-compact with two conical ends. Angenent \cite{angenent1992shrinking} introduced a compact, rotationally symmetric self-shrinker of genus one, an embedded torus with $\mathbb{S}^1\times \mathbb{S}^1$-symmetry.  
\medskip

Ilmanen \cite{ilmanen1998lectures} conjectured the existence of complete, non-compact, embedded self-shrinkers with arbitrary genus, one conical end, and dihedral symmetry. These surfaces resemble a plane and sphere desingularized along their intersection, forming handles that increase the genus. Kapouleas, Kleene, and Moller \cite{kapouleas2018mean} confirmed this conjecture for sufficiently large genus $g\geq g_0$. They desingularize the intersection of a plane and a sphere through the great circle of a sphere using Scherk-type surfaces, producing embedded self-shrinkers with genus $g$, one end asymptotic to a cone over a $\frac{2\pi}{2g+2}$-periodic curve in $\mathbb{S}^2$. These surfaces are invariant under the dihedral group $\mathbb{D}_{2g+2}$, of order
$4g+4$, reflecting rotations by $\frac{\pi}{g+1}$ and reflections across planes through the cone’s axis. Their method, however, does not yield low-genus examples.
\medskip

Concurrently, X. Nguyen \cite{nguyen2014construction} developed gluing techniques to approximate self-shrinkers by desingularizing two intersecting planes. While foundational, these constructions do not fully specify the resulting symmetries. Recently, Buzano, Nguyen, and Schulz \cite{buzano2025noncompact} and Ketover \cite{ketover2024self} employed min-max methods in Gaussian-weighted $\mathbb{R}^3$ to construct embedded self-shrinkers of any genus $g\geq 1$. Buzano, Nguyen, and Schulz \cite{buzano2025noncompact} produce surfaces, likely with one conical end, invariant under the dihedral group $D_{g+1}$ of order $2(g+1)$, corresponding to rotations by $\frac{2\pi}{g+1}$ and reflections through planes containing a central axis. Ketover \cite{ketover2024self} constructs non-compact self-shrinkers, likely one-ended for large $g$, invariant under the prismatic group $\mathbb{D}_{g+1}\times \mathbb{Z}_2$ of order $4(g+1)$. The $\mathbb{Z}_2$-action corresponds to a reflection, such as $(x,y,z)\rightarrow (x,y,-z)$, enhancing the dihedral symmetry with a mid-plane reflection. Ilmanen and White \cite{ilmanen2024fattening} prove the existence of these symmetric types of self-shrinkers which emerge as the tangent flow to the mean curvature flow initiated with a compact, connected, smoothly embedded surface of the genus $g+1$ in $\mathbb{R}^3$. As $g$ tends to infinity, these self-shrinkers converge to multiplicity 2 planes.
These self-shrinkers, particularly those with dihedral or prismatic symmetries, form the focus of our eigenvalue analysis. Their finite Gaussian area ensures a discrete spectrum for the drift Laplacian, and their symmetries constrain the eigenfunctions, facilitating our computation of the first eigenvalue in Section 6.

\section{First eigenvalue of $\mathcal{L}$}
Given $2\leq n\in \mathbb{N}$, the dihedral group $\mathbb{D}_n$ of order $2n$ is defined to be the discrete subgroup of Euclidean isometries acting on $\mathbb{R}^3$ generated by rotations $\rho_k:\mathbb{R}^3\rightarrow \mathbb{R}^3$ of angle $\pi$ around $n$ horizontal axes
$$
h_k:=\{ (r\cos \frac{k\pi}{n},r\sin \frac{k\pi}{n},0);\ r\in \mathbb{R}^3 \}
$$
for $k\in \{ 1,\cdots ,n\}$. Moreover, $\mathbb{D}_n$ contains the composition $\rho_2 \circ \rho_1$ that acts as a rotation by angle $\frac{2\pi}{n}$ around the vertical axis
$$
v:=\{ (0,0,r);\ r\in \mathbb{R}\}
$$
Let fix some notations:\\
Denote by $\Pi_{h_k}$ a plane containing the line $h_k$, $\Pi_H$ the horizontal plane $\{(x_1,x_2,0)\in \mathbb{R}^3 \}$, and $\Pi_{\langle v,h_k\rangle}$ the plane containing two intersecting lines $v$ and $h_k$. Note that, associated with any $\sigma \in \mathbb{D}_n$ there exists a plane (indeed many) that passes through the origin and is invariant only under the subgroup spanned by $\sigma$. Note that the horizontal plane is the unique plane that is invariant under the whole group $\mathbb{D}_n$.
\medskip

 Let $T$ be the domain in $\mathbb{R}^3$ enclosed by two planes $\Pi_{\langle v,h_k\rangle}$ and $\Pi_{\langle v,h_{k+1}\rangle}$. We say that $T$ is a fundamental domain of the group $\mathbb{D}_n$ in $\mathbb{R}^3$, in the sense that $T$ gives a cover of $\mathbb{R}^3$ by $n$ congruent domains under $\mathbb{D}_n$. Denote by $P:=\Sigma \cap T$ the fundamental patch of $\Sigma$, where $T$ is a fundamental domain of $\mathbb{R}^3$. Suppose that $\Sigma$ is invariant under $\mathbb{D}_n$, then the number of fundamental patches of $\Sigma$ is $n$.
 Let $H^+:=\{ p\in \mathbb{R}^3;\ \langle v,p\rangle >0\}$ and $H^-:=\{ p\in \mathbb{R}^3;\ \langle v,p\rangle <0\}$ to be the upper and lower half-space with respect to the vertical axis $v$. Set $T^+=T\cap H^+$ (resp. $P^+$) the domain in $H^+$ enclosed by two planes $\Pi_{\langle v,h_k\rangle}$ and $\Pi_{\langle v,h_{k+1}\rangle}$, then define $T^-$ (resp. $P^-$) accordingly. 

\begin{lem} \label{lem: lambda_1<1/2 implies f_1 invariant}
    Let $\mathbb{D}_n$ be the dihedral group of order $2n$. Suppose that $\Sigma$ is a complete, embedded, self-shrinker in $\mathbb{R}^3$ that is invariant under $\mathbb{D}_n$. If the first eigenvalue of the drift Laplacian $\mathcal{L}$ on $\Sigma$ is less than $\frac{1}{2}$, then the first eigenfunction must be invariant under $\mathbb{D}_n$.
\end{lem}
\begin{proof}
    Let $\sigma \in \mathbb{D}_n$ and denote by $\Pi_{\sigma}$ the set of the planes that pass through the origin and are invariant under $\sigma$. For $\Pi \in \Pi_{\sigma}$, denote by $\Pi^+$ and $\Pi^-$ the two distinct half-planes of $\Pi$ which are split by the line $h_{\sigma}$ satisfying $\sigma (p)=p$ for each $p\in h_{\sigma}$.
    Suppose that $\phi$ is an eigenfunction on $\Sigma$ of $\mathcal{L}$ corresponding with the first eigenvalue $\lambda_1$. Then $\phi \circ \sigma$ is also an eigenfunction with eigenvalue $\lambda_1$. Consider
    $$
    \psi(x):=\phi (x) - \phi \circ \sigma (x)
    $$
    If $\psi$ is the null function, then $\phi$ is invariant under $\sigma$. If $\psi \not \equiv 0$, then $\psi$ itself is an eigenfunction with eigenvalue $\lambda_1$. Furthermore, for any $p\in \Pi^+ \cap \Sigma$, $\psi (p)=\phi (p)-\phi (q)$, where $\sigma (p)=q \in \Pi^-$, and thus $\psi (q)=\phi (q)-\phi(p)=-\psi(p)$ for $q\in \Pi^-\cap \Sigma$ (here we use the fact that $\Sigma$ is invariant under $\mathbb{D}_n$). It follows that $\psi (\Pi^+\cap \Sigma)=-\psi(\Pi^-\cap \Sigma)$ and $\psi (h_{\sigma}\cap \Sigma)=0$. We claim that $\psi$, if  not identically zero in $\Pi^+ \cap \Sigma$ then it has a constant sign, either positive or negative, in $\Pi^+ \cap \Sigma$. That is because $\psi$ is continuous and $\psi \vert_{\Pi \cap \Sigma}$ is zero only in $h_{\sigma}\cap \Pi \cap \Sigma$. On the contrary if $\psi$ be zero elsewhere in $\Pi \cap \Sigma$, then $\phi$ is not one-to-one, and thus it takes the same value in two distinct domains of $\Sigma$. Then since it must change sign on $\Sigma$, $\phi$ has more than two nodal domains, a contradiction. Since the set $\Pi_{\sigma}$ covers $\mathbb{R}^3$, and also $\psi (\not \equiv 0)$ is continuous and not positive on $\Sigma$, there must be a plane $\Pi^0\in \Pi_{\sigma}$ such that $\psi({\Pi^0}^+ \cap \Sigma)=\psi({\Pi^0}^- \cap \Sigma)=0$. This implies that, $\psi$'s nodal set contains $\Pi^0 \cap \Sigma$.
    By Theorem \ref{thm: 2-piece prop}, $\Pi^0$ divides $\Sigma$ into two connected pieces, and by Courant's nodal domain theorem \ref{thm: courant}, $\psi$ has precisely two nodal domains. Therefore, $\psi$ vanishes only on $\Pi^0 \cap \Sigma$. Let $D_1,\ D_2$ be the components of $\Sigma \backslash \Pi^0$ such that $\psi$ is positive on $D_1$ and negative on $D_2$. We can find a linear function on $\mathbb{R}^3$, $\xi = a_1x_1+a_2x_2+a_3x_3$ that vanishes on $\Pi^0$ and is positive on $D_1$. Recall that, $x_i$ is an eigenfunction on $\Sigma$ of $\mathcal{L}$ with eigenvalue $\frac{1}{2}$. Since $\lambda_1 < \frac{1}{2}$, $\xi$ is orthogonal to $\xi$ on $\Sigma$. Since $\psi$ and $\xi$ have the same sign on $D_1 \cup D_2$, we have $\int_{\Sigma}\psi \xi >0$, which contradicts the orthogonality of $\psi$ and $\xi$. Therefore, $\psi$ must vanish on $\Sigma$. Since $\sigma$ is an arbitrary element of $\mathbb{D}_n$, the proof is complete.
\end{proof}
\begin{thm}
    Suppose that $\Sigma$ is a complete, embedded self-shrinker in $\mathbb{R}^3$ which is invariant under the group $\mathbb{D}_n$. Then the first eigenvalue of $\mathcal{L}$ on $\Sigma$ equals $\frac{1}{2}$.
\end{thm}
\begin{proof}
    Suppose that $\lambda_1 <\frac{1}{2}$, let $\phi$ be an eigenfunction of $\mathcal{L}$ with eigenvalue $\lambda_1$ on $\Sigma$, and $N\subset \Sigma$ the nodal set of $\phi$. Then, by Lemma \ref{lem: lambda_1<1/2 implies f_1 invariant} $\phi$ must be invariant under $\mathbb{D}_n$. Let $T$ be a fundamental domain of $\mathbb{R}^3$, and $P=\Sigma \cap T$ be the corresponding fundamental patch of $\Sigma$.
    First we claim that $N$ contains an interior point of $P$ and $P\backslash N$ has at least two connected components. Suppose that $\bar{P} \cap N\subset \partial P$. Then by the $\mathbb{D}_n$-invariance of $\phi$, the Courant's nodal domain theorem \ref{thm: courant}, and the two-piece property \ref{thm: 2-piece prop} we conclude that $N=\Sigma \cap \Pi$ for some plane $\Pi$ which passes through the origin. By the orthogonality argument as used in Lemma \ref{lem: lambda_1<1/2 implies f_1 invariant}, we get a contradiction. Hence $P\backslash N$ is not connected. Let $\Pi_{h_1}$, $\Pi_{h_2}$ be the faces of $T$, and let $D$ be a connected component of $P\backslash N$. 
    \medskip
    
    First suppose that $\partial D$ is disjoint from the line $h_1$, and suppose that $\hat{D}$ is the image of $D$ rotated by $\pi$ around $h_1$, that is $\hat{D}=\rho_1(D)$. Denote by $D_1,D_2,D_3$ the components of $\Sigma \backslash N$ containing $D, \hat{D}$ and intersecting $h_1$, respectively. We claim that $D_1,D_2,D_3$ all are distinct. Suppose on the contrary that $D_1$ and $D_3$ are identical. Let $\hat{T}$ be the image of $T$ obtained by rotating by $\pi$ around $h_1$, and set $T_2=\hat{T}\cup T$. Choose $p\in D$ and let $\hat{p}=\rho_1(p) \in \hat{D}$ be its rotated image. By the assumption, there is a curve $\gamma$ in $\Sigma$ connecting $p$ to $\hat{p}$ which is disjoint from $N$. Let $F$ be the normal closure of the subgroup of $\mathbb{D}_n$ which is generated by the rotations by $\pi$ around the axis different from $h_1$. $F$ then is a subgroup of $\mathbb{D}_n$ of index 2 whose fundamental domain is $T_2$. Define $\tilde{\gamma}=\{ \rho (\gamma)\cap \bar{T_2};\ \rho \in F\}$. Then $\tilde{\gamma}$ is a curve in $\bar{T_2}$ connecting $p$ to $\hat{p}$ and disjoint from $N\cap T_2$, that is impossible. Hence $D_1$ and $D_3$ are distinct components. A similar argument applies to conclude that $D_1$ and $D_2$ are also distinct.
    \medskip
    
     Now suppose that, $\partial D $ intersects both $h_1$ and $h_2$. $\Pi_H$ divides $P$ into two connected components $P^+$ and $P^-$. First assume that $D\subset P^+$, then note that $D\neq P^+$, otherwise, by a simple argument and an application of the Courant nodal domain theorem we obtain $N=\Sigma \cap \Pi_H$ and we thus get a contradiction with $\lambda_1<\frac{1}{2}$. Let $\hat{D}=\rho_1(D)$, and denote by $D_1,D_2,D_3$ the components of $\Sigma \backslash N$ containing $D, \hat{D}$ and $P\backslash \{ D\cup N\}$, respectively. We claim that $D_1, D_2, D_3$ are all distinct. Let $\hat{T}=\rho_1(T)$ and $ \hat{P}=\rho_1(P)$, and note that $\hat{P}^+=\rho_1(P^-)$ and $\hat{P}^-=\rho_1(P^+)$ (this is because $h_1\subset \Pi_H$). Then, $\hat{D} \subset \hat{P}^-$. Suppose on the contrary that $D_1$ and $D_2$ are identical. Set $T_2=T\cup \hat{T}$. Choose $p\in D$ and let $\hat{p}=\rho_1(p) \in \hat{D}$ be its image. By the assumption, there is a curve $\gamma$ connecting $p$ to $\hat{p}$ which is disjoint from $N$. Consider $\tilde{\gamma}=\{ \rho(\gamma) \cap \bar{T}_2;\ \rho\in F\}$, where $F$ is as defined above. Then $\tilde{\gamma}$ is a curve in $\bar{T}_2$ connecting $p$ to $\hat{p}$ and disjoint from $N\cap T_2$. But this is impossible. Hence $D_1$ and $D_2$ are distinct components. The same argument can be used to conclude that $D_2$ and $D_3$ are also distinct.
     \medskip
     
     Finally, suppose that $\partial D$ intersects both $h_1$ and $h_2$, and also $D$ is not a subset of $P^+$ or $P^-$. Similarly we assume $\partial D^c$ intersects both $h_1$ and $h_2$, where $D^c=P\backslash \{N\cup D\}$ is not a subset of $P^+$ or $P^-$, as well. Since $\phi$ is invariant under $\mathbb{D}_n$, $N$ meets each line $h_k$, $k=1,\cdots ,n$.
     Set $\hat{D}=\rho_1(D)$, $\hat{D^c}=\rho_1(D^c)$. Suppose that $\phi \vert_D>0$ and $\phi \vert_{D^c} <0$, since $\phi$ is $\mathbb{D}_n$-invariant then $\phi \vert_{\hat{D}}>0$ and $\phi \vert_{\hat{D^c}} <0$. Obviously $N$ meets the vertical line $v$ both in $\Sigma^-$ and $\Sigma^+$. In fact, since $\phi$ has the opposite sign on $D$ and $\hat{D}$, $N$ meets $v$ at least $2n$ times in $\Sigma^-$ and $2n$ times in $\Sigma^+$ forming an equiangular system, see Theorem \ref{thm: Props. of nodal lines}. Recall that $h_k$s are in the plane $\Pi_H$ and $v$ is orthogonal to $\Pi_H$. Hence $N$ divides $\Sigma$ into more than two components. Therefore in all possible cases, $\phi$ has at least three nodal domains, contradicting the Courant's nodal domain theorem, thus $\lambda_1=\frac{1}{2}$.
\end{proof}

\begin{cor}
    The first eigenvalue of the drift Laplacian on self-shrinkrs constructed by  Kapouleas-Kleene-Moller \cite{kapouleas2018mean}, Buzano-Nguyen-Schulz \cite{buzano2025noncompact}, Ketover \cite{ketover2024self}, and Ilmanen and White \cite{ilmanen2024fattening} is equal to $\frac{1}{2}$.
\end{cor}
\subsection*{Acknowledgment} I wish to thank my advisor, Jacob Bernstein, for suggesting the problem and for the helpful comments.

\end{document}